\documentclass{scrartcl}

\usepackage[latin1]{inputenc}
\usepackage[hang]{subfigure}

\usepackage{amsmath}
\usepackage{amssymb}
\usepackage{amsthm}
\usepackage{latexsym}
\usepackage{bbm}

\usepackage{tikz}
\usepackage{xspace}
\graphicspath{{images/}}
\usepackage[style=mystyle, bibstyle=mybibstyle, maxnames=7, useauthor]{biblatex}
\usepackage{hyperref}

\theoremstyle{plain}
\newtheorem{theorem}{Theorem}[section]
\newtheorem{proposition}[theorem]{Proposition}
\newtheorem{lemma}[theorem]{Lemma}
\newtheorem{corollary}[theorem]{Corollary}

\theoremstyle{definition}

\newtheorem{definition}[theorem]{Definition}

\newtheorem{example}[theorem]{Example}

\newcommand{\Ie}{\textit{I.e.,}\xspace}
\newcommand{\ie}{\textit{i.e.,}\xspace}
\newcommand{\Cf}{\textit{Cf.}\ }
\newcommand{\cf}{\textit{cf.}\ }

\newcommand{\eg}{\textit{e.g.}\xspace}

\newcommand{\R}{\ensuremath{\mathbb R}}		
\newcommand{\T}{\mathbb{T}}
\newcommand{\TP}{\ensuremath{\mathbb {TP}}}	

\newcommand{\mc}{\mathcal}

\newcommand{\tiff}{if and only if\xspace}
\newcommand{\twlog}{without loss of generality\xspace}

\newcommand{\tom}{tropical oriented matroid\xspace}
\newcommand{\toms}{tropical oriented matroids\xspace}
\newcommand{\tphp}{tropical pseudohyperplane\xspace}
\newcommand{\tphps}{tropical pseudohyperplanes\xspace}

\newcommand{\thp}{tropical hyperplane\xspace}
\newcommand{\thps}{tropical hyperplanes\xspace}

\newcommand{\thpas}{\thp arrangements\xspace}

\newcommand{\atphp}{arrangement of \tphps}
\newcommand{\tphpa}{\tphp arrangement\xspace}
\newcommand{\tphpas}{\tphp arrangements\xspace}

\newcommand{\mixsd}{mixed subdivision\xspace}		
\newcommand{\mixsds}{mixed subdivisions\xspace}	

\newcommand{\ndmixsd}{\mixsd of $\dilsimp n{d-1}$\xspace}
\newcommand{\dnmixsd}{\ndmixsd}
\newcommand{\dnmixsds}{\mixsds of $\dilsimp n{d-1}$\xspace}

\newcommand{\defn}[2][!*!,!]{\emph{#2}}

\newcommand{\coloneq}{\mathrel{\mathop:}=}
\newcommand{\eqcolon}{=\mathrel{\mathop:}}

\newcommand{\simplex}{\triangle}			
\newcommand{\nsimplex}[1]{\simplex^{#1}}	

\newcommand{\dual}[1]{#1 ^{*}}				
\DeclareMathOperator{\card}{\#\!}			
\DeclareMathOperator{\rank}{rank}			
\newcommand{\isom}{\cong}					

\newcommand{\BIGOP}[1]{\mathop{\mathchoice%
{\raise-0.22em\hbox{\huge $#1$}}%
{\raise-0.05em\hbox{\Large $#1$}}{\hbox{\large $#1$}}{#1}}}

\newcommand{\deletion}[2]{#1_{\setminus #2}}	
\newcommand{\contraction}[2]{#1_{/#2}}			

\newcommand{\restr}[2]{#1|_{#2}}				
\newcommand{\compgrop}{%
	\mathrel{\vcenter{\offinterlineskip
	\hbox{C\,}\vskip-1ex\hbox{\,\,G}}}}
\newcommand{\compgr}[2]{\compgrop_{#1,#2}}		
\newcommand{\omitted}[1]{\widehat{#1}}			

\newcommand{\dilsimp}[2]{{#1}\nsimplex{{#2}}}		
\newcommand{\simpprod}[2]{\nsimplex{#1}\times\nsimplex{#2}}	

\begin{document}
\author{Silke Horn\\TU Darmstadt\\\url{shorn@opt.tu-darmstadt.de}}
\title{A Topological Representation Theorem\\for Tropical Oriented Matroids: Part I}

\maketitle

\begin{abstract}
Tropical oriented matroids were defined by Ardila and Develin in 2007 in analogy to (classical) oriented matroids. 
In this paper we present one tropical analogue for the Topological Representation Theorem.  

\end{abstract}

\section{Introduction}
Oriented matroids abstract the combinatorial properties of arrangements of real hyperplanes and are ubiquitous in combinatorics. In fact, an arrangement of $n$ (oriented) real hyperplanes in $\R^d$ induces a regular cell decomposition of $\R^d$. Then the covectors of the associated oriented matroid encode the position of the points of $\R^d$ (respectively, the cells in the subdivision) relative to the each of the hyperplanes in the arrangement. It turns out though that there are oriented matroids which cannot be realised by any arrangement of hyperplanes.
The famous Topological Representation Theorem by Folkman and Lawrence \cite{Folkman/Lawrence} (see also \cite{BLSWZ}), however, states that every oriented matroid can  be realised as an arrangement of PL-\emph{pseudo}hyperplanes. 

In this paper, we will study a \emph{tropical} analogue of oriented matroids.

\bigskip
Tropical geometry is the algebraic geometry over the tropical semiring $(\bar\R\coloneq\R\cup\{\infty\},\oplus,\otimes)$, where \[\oplus:\bar\R\times\bar\R\to\bar\R:a\oplus b\coloneq \min\{a,b\}\qquad \text{and}\qquad \otimes:\bar\R\times\bar\R\to\bar\R: a\otimes b\coloneq a+b\] are the tropical addition and multiplication. 
A \thp is the vanishing locus of a linear tropical polynomial, \ie the set of points $x$ where the minimum $p(x)=\BIGOP\oplus(a_i\otimes x_i)$ is attained at least twice.

Tropical geometry can be seen as a piecewise linear image of classical algebraic geometry and has in the past years received attention since there are often strong relationships between classical problems and their tropical analogues, see \eg \cite{Ardila/Billey, Ardila/Klivans, Develin/Sturmfels, Mikhalkin06}.

Moreover, tropical geometry has far reaching connections to other objects of discrete geometry:
By Develin and Sturmfels \cite{Develin/Sturmfels} \emph{regular} subdivisions of $\simpprod{n-1}{d-1}$ are dual to arrangements of $n$ \thps in $\T^{d-1}$. See Figure \ref{fig:mixsd2thp} for an illustration.
By the \defn{Cayley Trick} (\cf Huber, Rambau and Santos \cite{HRS00}) subdivisions of $\simpprod{n-1}{d-1}$ are in bijection with \mixsds of $\dilsimp n{d-1}$.

\medskip 
Here we will take a more combinatorial point of view and study the relationship of arrangements of \thps and \dnmixsds: 
Combinatorially, a \thp in $\T^{d-1}$ is just the (co\-dimen\-sion-$1$-skeleton of the) polar fan of the $(d-1)$-dimensional simplex $\nsimplex{d-1}$. 
For a $(d-2)$-dimensional tropical hyperplane $H$ the $d$ connected components of $\TP^{d-1}\setminus H$ are called the \defn{(open) sectors} of $H$ -- they form the analogues to the sides $+/-$ of the classical oriented hyperplanes.

As in the classical situation, an arrangement of $n$ \thps in $\T^{d-1}$ induces a cell decomposition of $\T^{d-1}$ and each cell can be assigned a \defn{type} or \defn{tropical covector} that describes its position relative to each of the \thps. To be precise, the point $p$ is assigned the type $A=(A_1,\ldots,A_n)$ where $A_i\subseteq[d]$ denotes the set of closed sectors of the $i$-th \thp in which $p$ is contained. See Figure \ref{fig:thpa} for an illustration in dimension $2$.

\begin{figure}[t]
\centering
\subfigure[A (regular) \mixsd of $\dilsimp22$.]
{\label{fig:mixsd2thp_a}\includegraphics[width=3.5cm]{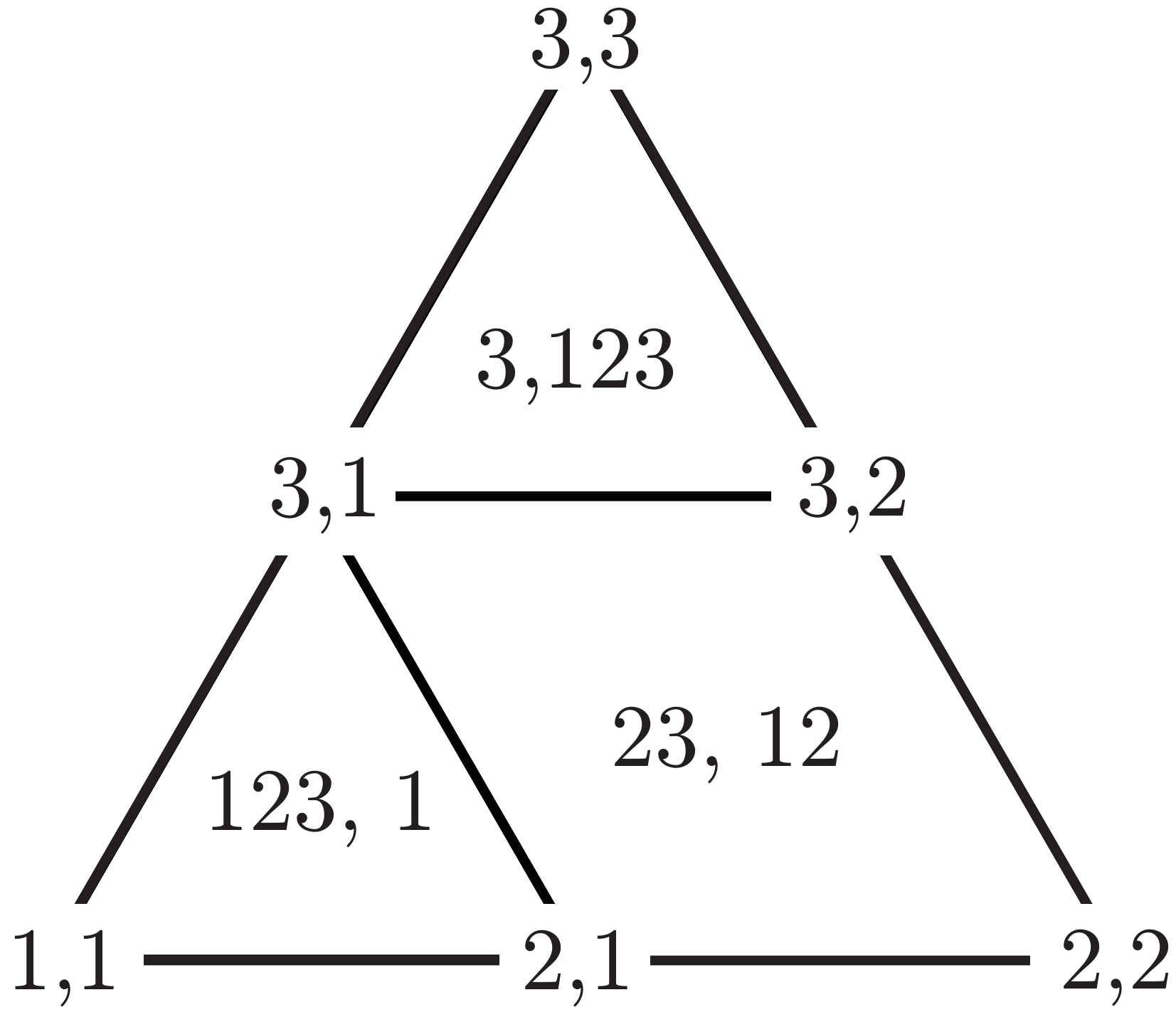}}\qquad
\subfigure[The Poincaré dual \mbox{of \subref{fig:mixsd2thp_a}}.]
{\label{fig:mixsd2thp_b}\includegraphics[width=3.5cm]{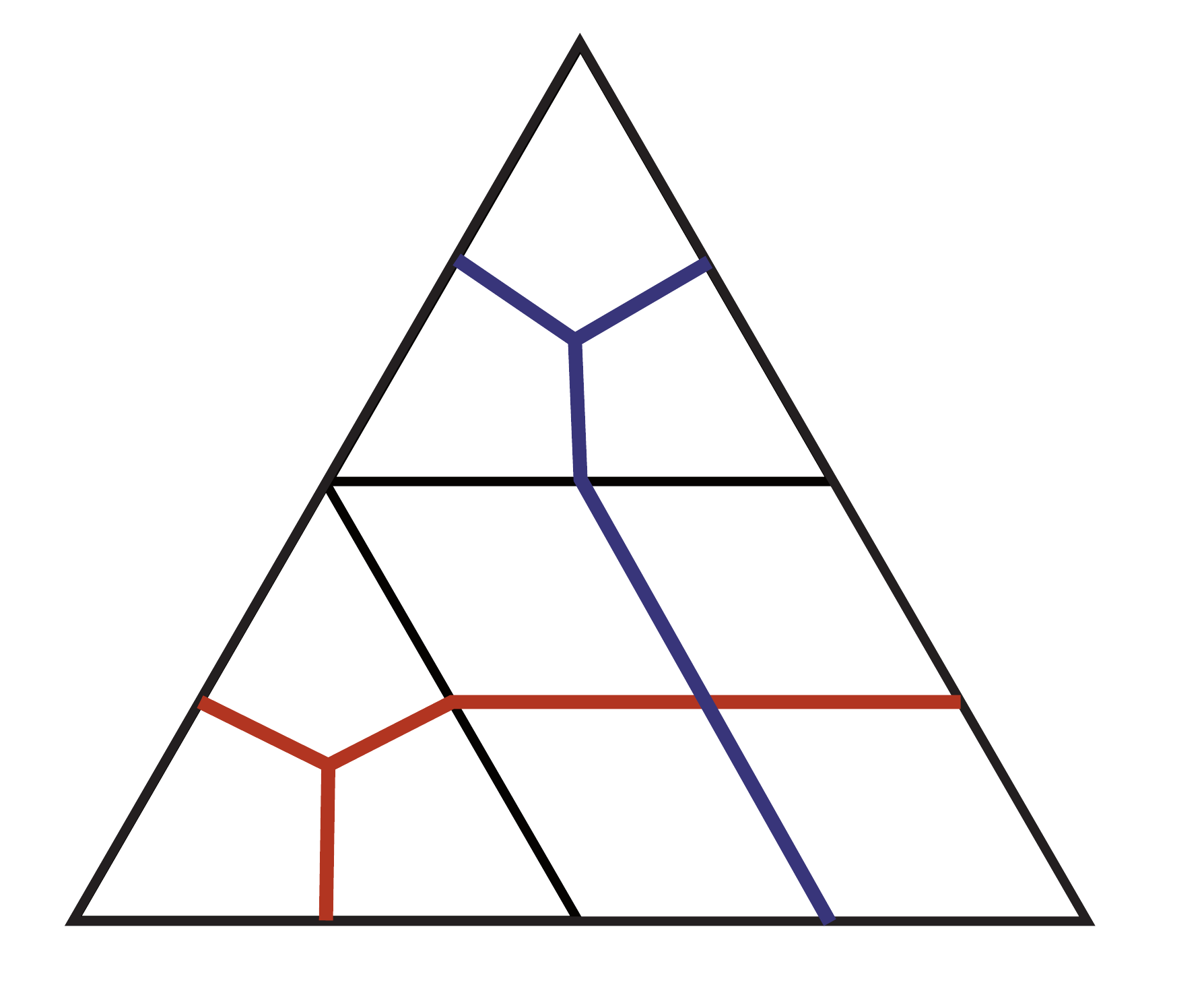}}\qquad
\subfigure[An arrangement of \thps.\label{fig:thpa}]
{\includegraphics[width=3.5cm]{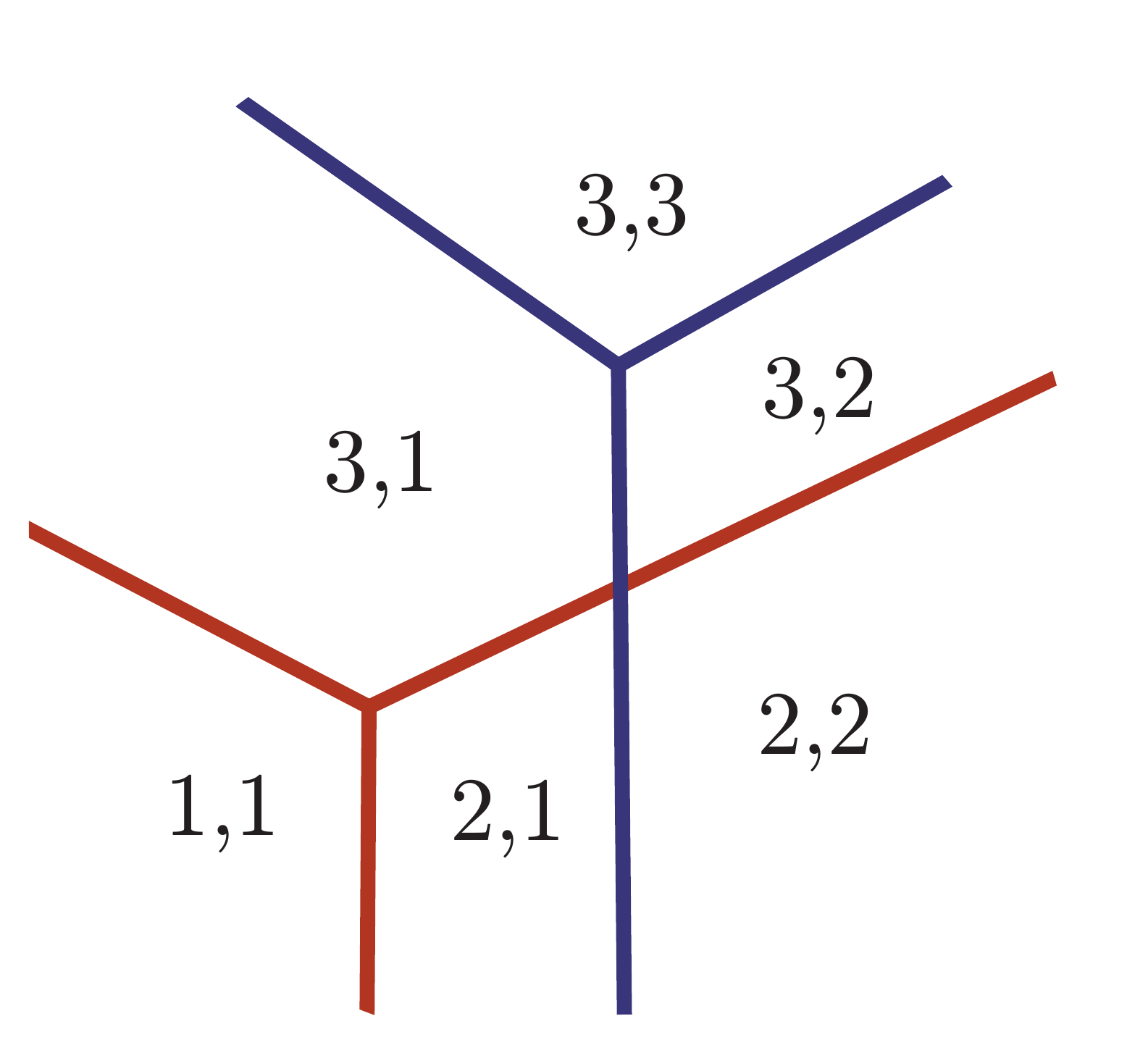}}
\caption{The correspondence between \mixsds and \tphpas.}
\label{fig:mixsd2thp}
\end{figure}

By \cite[Theorem 6.3]{Ardila/Develin}, the types of a \tom with parameters $(n,d)$  yield a subdivision of $\simpprod{n-1}{d-1}$.
They also conjecture this to be a bijection, \ie they conjecture that the types of the cells in any \mixsd of $\dilsimp n{d-1}$ are the types of a \tom with parameters $(n,d)$. 

In Oh and Yoo \cite{Oh/Yoo} the conjecture is proven for \emph{fine} \mixsds; in \cite{toprep2} we generalise this result to all \dnmixsds.

\bigskip
In this paper we introduce arrangements of \tphps (see Definition \ref{def:tphpa1}) and prove one tropical analogue to the Topological Representation Theorem for (classical) oriented matroids \cite{Folkman/Lawrence}.

A \defn{\tphp} is basically a set which is PL-homeomorphic to a \thp (see also Definition \ref{def:tphp}). The challenging part is the definition of arrangements of such: We have to impose restrictions on the intersections of the pseudohyperplanes in the arrangement. In the classical framework, the intersections of the hyperplanes in the arrangement have to be homeomorphic to linear hyperplanes (of smaller dimension). In the tropical world, however, this approach is not feasible, since intersections of tropical hyperplanes are no longer homeomorphic to tropical hyperplanes (but have a very complicated geometry). We choose a different approach instead. We employ the fact that the covectors of \toms define \dnmixsds and impose restrictions on the cell decomposition induced by the \tphps in the arrangement.  
We this definition we prove the main result of this paper:

\begin{theorem}[{Topological Representation Theorem, \cf Theorem \ref{thm:toprep1}}]
Every \tom  can be realised by an \atphp.
\end{theorem}

The paper is organised as follows: In Section \ref{sec:toms} we briefly review the definition of \toms as given in \cite{Ardila/Develin}. Section \ref{sec:mixsds} is dedicated to \mixsds of dilated simplices. In particular, we prove some results strengthening the close relationship to \toms that are, to the best of our knowledge, not yet in the literature. Moreover, we show  that a \dnmixsd is uniquely determined by (the types of) its $0$-cells (Proposition \ref{prop:reconstruct_mixsd}). In Section \ref{sec:toprep} we introduce \tphpas and prove a first version of the Topological Representation Theorem. Another version and a second (equivalent) definition of \tphpas are given in \cite{toprep2}.

\bigskip
A joint extended abstract \cite{toprep_fpsac} of this and \cite{toprep2} has been presented at FPSAC 2012. Moreover, the results in this paper are also contained in \cite{mydiss}.

\section{Tropical Oriented Matroids}\label{sec:toms}

The following definitions are analogous to those in \cite{Ardila/Develin}.

\begin{definition}
\label{def:nd_type}
For $n,d\geq1$ an \defn[(n,d)-type]{$(n,d)$-type} is an $n$-tuple $(A_1,\ldots,A_n)$ of non-empty subsets \mbox{of $[d]$}.

An $(n,d)$-type $A$ can be represented as a subgraph $K_A$ of the complete bipartite graph $K_{n,d}$: Denote the vertices of $K_{n,d}$ by $N_1,\ldots,N_n,D_1,\ldots,D_d$. Then the edges of $K_A$ are $\{\{N_i,D_j\}\mid j\in A_i\}$.
\end{definition}

For convenience we will write sets like $\{1,2,3\}$ as $123$ throughout this article.

In particular, the types of points relative to an arrangement of $n$ tropical hyperplanes of dimension $d-1$ as described above are $(n,d)$-types.

\medskip
A \defn{refinement} of an $(n,d)$-type $A$ with respect to an ordered partition $P=(P_1,\ldots,P_k)$ of $[d]$ is the $(n,d)$-type $B=\restr A P$ where $B_i=A_i\cap P_{m(i)}$ and $m(i)$ is the smallest index where $A_i\cap P_{m(i)}$ is non-empty for each $i\in[n]$. A refinement is \defn[total!refinement]{total} if all $B_i$ are singletons.

\smallskip
Given $(n,d)$-types $A$ and $B$, the \defn{comparability graph} $\compgr AB$ is a multigraph with node set $[d]$. For $1\leq i\leq n$ there is an edge for every $j\in A_i,k\in B_i$. This edge is undirected if $j,k\in A_i\cap B_i$ and directed $j\rightarrow k$ otherwise. (We consider the comparability graph as a graph without loops.)
Note that there may be several edges (with different directions) between two nodes.

A \defn[]{directed path} in the comparability graph is a sequence $e_1,e_2,\ldots,e_k$ of incident edges at least one of which is directed and all directed edges of which are directed in the ``right'' (\ie the same) direction. A \defn[]{directed cycle} is a directed path whose starting and ending point agree. The graph is \defn[]{acyclic} if it contains no directed cycle.


\begin{definition}
\label{def:trop_or_mat}
A \defn{tropical oriented matroid} $M$ (with parameters $(n,d)$) is a collection of $(n,d)$-types \index{type!in a tropical oriented matroid|see{$(n,d)$-type}} which satisfies the following four axioms:
\begin{itemize}
\item\defn[boundary axiom]{Boundary}: For each $j\in[d]$, the type $(j,j,\ldots,j)$ is in $M$.
\item\defn[comparability axiom]{Comparability}: The comparability graph $\compgr AB$ of any two types $A,B\in M$ is acyclic.
\item\defn[elimination axiom]{Elimination}: If we fix two types $A, B\in M$ and a position $j\in[n]$, then there exists a type $C$ in $M$ with $C_j=A_j\cup B_j$ and $C_k\in\{A_k,B_k,A_k\cup B_k\}$ for $k\in[n]$.
\item\defn[surrounding axiom]{Surrounding}: If $A$ is a type in $M$, then any refinement of $A$ is also in $M$.
\end{itemize}
We call $d\eqcolon\rank M$ the \defn[rank!of a \tom]{rank} and $n$ the \defn[size!of a \tom]{size} of $M$. 
\end{definition}

\begin{example}
By \cite[Theorem 3.6]{Ardila/Develin} the set of types of an arrangement of $n$ tropical hyperplanes in $\T^{d-1}$ is a \tom with parameters $(n,d)$.
\end{example}
We  call \toms coming from an arrangement of \thps \defn{realisable}. Recall that by Develin and Sturmfels \cite{Develin/Sturmfels} realisable \toms are in bijection with \emph{regular} \dnmixsds.

\smallskip
The axiom system was built to capture the features of the set of types in \thpas and thus the axioms have geometric interpretations: 

The \emph{boundary axiom} ensures that all tropical hyperplanes in the arrangement are embedded correctly into $\TP^{d-1}\isom\nsimplex{d-1}$.
The \emph{surrounding axiom} describes what the neighbourhood of a point of type $A$ (or equivalently, the star of the cell $A$ in the cell complex) looks like. 
The \emph{elimination axiom} describes the intersection of a tropical line segment from $A$ to $B$ with the $j$-th \thp.
Finally, the \emph{comparability axiom} ensures that we can declare a ``direction from $A$ to $B$''. Each position puts certain constraints on the direction vector, which may not contradict one another.

\begin{definition}\label{def:tom_props}
The \defn[dimension!of a type]{dimension} of an $(n,d)$-type $A$ is the number of connected components of $K_A$ minus $1$.
A \defn[vertex!in a \tom]{vertex} is a type of dimension $0$, an \defn[edge!in a \tom]{edge} a type of dimension $1$ and a \defn{tope} a type of full dimension $d-1$, \ie each tope is an $n$-tuple of singletons.


For two types $A,B$ we  write $A\supseteq B$ if $A_i\supseteq B_i$ for each $i\in[n]$. Moreover, we define the intersection $A\cap B\coloneq(A_1\cap B_1,\ldots,A_n\cap B_n)$ and union $A\cup B\coloneq (A_1\cup B_1,\ldots, A_n\cup B_n)$.

A type $A$ in a \tom $M$ is \defn[bounded!type]{bounded} if all elements of $[d]$ appear in $A$ and \defn[unbounded type]{unbounded} otherwise.
\end{definition}

A realisable \tom is in general position \tiff the corresponding arrangement of \thps is so.
Moreover, for a realisable \tom, the bounded types correspond to the bounded cells in the cell decomposition of $\T^{d-1}$.

\begin{definition}[\Cf {\cite[Propositions 4.7 and 4.8]{Ardila/Develin}}] \label{def:tom_deletion_contraction}
 Let $M$ be a \tom with parameters $(n,d)$.
\begin{enumerate}
\item For $i\in[n]$ the \defn[deletion!in a \tom]{deletion} $\deletion Mi$ consisting of all $(n-1,d)$-types which arise from types of $M$ by deleting coordinate $i$ is a \tom with parameters $(n-1,d)$.
\item For $j\in[d]$ the \defn[contraction!in a \tom]{contraction} $\contraction Mj$ consisting of all types of $M$ that do not contain $j$ in any coordinate is a \tom with parameters $(n,d-1)$.
\end{enumerate}
\end{definition}


\section{Mixed Subdivisions}\label{sec:mixsds}

Given two sets $X,Y$ their \defn{Minkowski sum} $X+Y$ is given by $X+Y\coloneq\{x+y\mid x\in X,y\in Y\}$.

\begin{definition}\label{def:mixed_sd}
Let $P_1,\ldots,P_k\subset \R^n$ be (full-dimensional) convex polytopes. Then a polytopal subdivision  $\{Q_1,\ldots,Q_s\}$ of $P\coloneq\sum P_i$ is a \defn{mixed subdivision} if it satisfies the following conditions:
\begin{enumerate}
\item Each $Q_i$ is a Minkowski sum $Q_i=\sum\limits_{j=1}^k F_{i,j}$, where $F_{i,j}$ is a face of $P_j$.
\item For $i,j\in[s]$ we have that $Q_i\cap Q_j=(F_{i,1}\cap F_{j,1})+\ldots+(F_{i,k}\cap F_{j,k})$.
\end{enumerate}
\end{definition}

Note that this definition can easily be generalised for polytopes which are not full-dimensional.

Let $S,S'$ be \mixsds of $\dilsimp n{d-1}$. Then we say that $S'$ is a \defn[refinement!of a \mixsd]{refinement} of $S$  if for every cell $C'\in S'$ there is a cell $C\in S$ such that $C'\subseteq C$.
This defines a partial order on the set of \mixsds of $\dilsimp n{d-1}$.
A \mixsd is \defn[fine!mixed subdivision]{fine} if there is no \mixsd refining it. 
By Santos \cite[Proposition 2.3]{Santos03} this is equivalent to the condition that for every cell $B=\sum B_i$ all the $B_i$ lie in mutually 
independent affine subspaces (and this is satisfied \tiff $\dim B=\sum\dim B_i$).



\subsection{Mixed Subdivisions of \boldmath{$\dilsimp n{d-1}$}}

We are interested in the case of \mixsds where $P_i=\nsimplex{d-1}$ for each $i$. Then $\sum P_i=\dilsimp n{d-1}$ is a dilated simplex. By Ardila and Develin \cite[Theorem 6.3]{Ardila/Develin} the types of a \tom with parameters $(n,d)$ yield a \dnmixsd. A \tom is in general position if and only if its \mixsd is fine.

If $Q=\sum_{i=1}^k F_i$, where $F_i\subset[d]$ is a cell in such a \mixsd then we  call $(F_1, F_2,\ldots,F_k)$ its \defn[type!of a cell in a \mixsd]{type} and denote it by $\mc T_Q$. Note that this is an $(n,d)$-type as defined in Definition \ref{def:nd_type}.
Conversely, if we are given an $(n,d)$-type $A$ then this corresponds to a unique cell inside $\dilsimp n{d-1}$, which we denote by $\mc C_A$.

In general, we call a cell corresponding to an $(n,d)$-type, \ie a Minkowski sum of $n$ faces of $\nsimplex{d-1}$, a \defn{Minkowski cell}.

\smallskip
To avoid confusion with the vertices of \toms, we  speak of the $0$-dimensional cells of a \mixsd as \defn{topes}.

\medskip
We now establish some properties of \dnmixsds\ -- or more generally about $(n,d)$-types. Note that since we can describe the Minkowski cells in a \dnmixsd in terms of $(n,d)$-types, we can transfer properties of \toms (such as the boundary, surrounding, comparability or elimination property) as defined in Section \ref{sec:toms} to \dnmixsds.

\begin{lemma}\label{lem:refinement_iff_acyclic}
Let $A,B$ be two $(n,d)$-types with $A\subseteq B$. Then $A$ is a refinement of $B$ \tiff $\compgr AB$ is acyclic.
\end{lemma}
Note that we do not assume that the types in this lemma are contained in a \tom.
In particular, there is a \tom containing both $A$ and $B$ \tiff $\compgr AB$ is acyclic.
\begin{proof}
First assume that $\compgr AB$ is acyclic. Let $G$ be the directed graph obtained from $\compgr AB$ by contracting all undirected edges. This is well-defined and acyclic since $\compgr AB$ is acyclic. We will label the vertices of $G$ by the according subsets of $[d]$. Let $P=(P_1,\ldots,P_\ell)$ be a linear extension of the partial order on the vertices of $G$ that is defined by the edges. This process is illustrated in Figure \ref{fig:refinement_iff_acyclic}. We will now argue that $\restr BP=A$. Indeed by the definition of refinements, $(\restr BP)_i$ contains all elements of $B_i$ which come first in $P$. 
Since $A_i\subseteq B_i$,  in $\compgr AB$ every element of $A_i$ has an outgoing edge to each element of $B_i-A_i$. Hence in $P$ the elements of $A_i$ come before the elements of $B_i-A_i$. Moreover, the elements of $A_i$ form a clique in $\compgr AB$ and are thus contained in the same $P_i$. This shows that $(\restr BP)_i=A_i$ for each $i\in[n]$.
\begin{figure}[h]
\centering
\begin{tikzpicture}
\def\k{1};
\node (1) at (0,0) {$1$};
\node (2) at (-\k,-\k) {$2$};
\node (3) at (0,-2*\k) {$3$};
\node (4) at (1.5*\k,-2*\k) {$4$};
\node (5) at (2.5*\k,-\k) {$5$};
\node (6) at (1.5*\k,0) {$6$};

\draw (1)--(2)--(3)--(1);
\draw [->] (1)--(6);
\draw [->] (3)--(4);
\draw [->] (4)--(5);
\draw [->] (4)--(6);
\draw (5)--(6);
\end{tikzpicture}
\qquad\qquad
\begin{tikzpicture}
\def\k{1};
\node (123) at (-.7*\k,-\k) {$123$};
\node (4) at (\k,-2*\k) {$4$};
\node (56) at (\k,0) {$56$};
\draw [->] (123)--(4);
\draw [->] (123)--(56);
\draw [->] (4)--(56);
\end{tikzpicture}
\qquad\qquad
\begin{tikzpicture}
\def\k{1}; 
\node (P) at (0,-\k) {$P=123,4,56$};
\node (0) at (0,0) {};
\node (1) at (0,-2*\k) {};
\end{tikzpicture}
\caption[Illustration for the proof of Lemma \ref{lem:refinement_iff_acyclic}.]{Assume that in the proof of Lemma \ref{lem:refinement_iff_acyclic} we have $A=(123,1,3,4,56), B=(123,16,34,456,56)$. Then $A\subseteq B$ and $\compgr AB$ is the graph on the left. Then by contracting all undirected edges we obtain the graph $G$ drawn in the center. By fixing a linear extension of this (in fact, there is only one in this example) we get the ordered partition of $[6]$ on the right hand side. Moreover, one easily verifies that indeed $A=\restr BP$.}
\label{fig:refinement_iff_acyclic}
\end{figure}
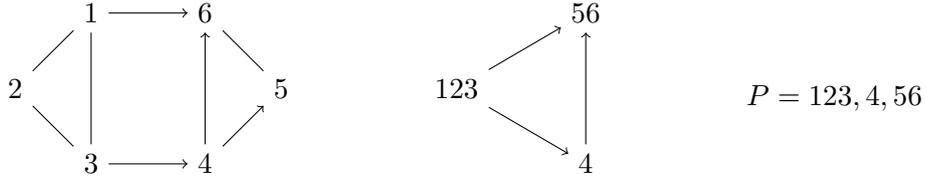

Conversely, assume that $A=\restr BP$ for some ordered partition $P$ of $[d]$. Consider the graph $H=([d],E)$ with an undirected edge $\{i,j\}$ for each $i,j\in P_a$ and a directed edge $i\to j$ whenever $i\in P_a, j\in P_b$ with $a<b$. Then clearly $H$ is acyclic. We now show that $\compgr AB$ is a subgraph of $H$, which completes the claim. Indeed let $i,j\in[d]$. If $\compgr AB$ has an undirected edge $\{i,j\}$ then there is $k\in[n]$ such that $i,j\in A_k\cap B_k$ and hence there is $P_\ell$ such that $i,j\in P_\ell$. 
On the other hand, if $\compgr AB$ has a directed edge $i\to j$ then there is $k\in[n]$ such that $i\in A_k\cap B_k$ but $j\in B_k-A_k$. If we choose $a,b$ such that $i\in P_a$ and $j\in P_b$ then we must have $a<b$.
\end{proof}

\begin{lemma}\label{lem:intersection_of_types} 
Let $A,B$ be two types in a \mixsd $S$ of $\dilsimp{n}{d-1}$. Then their intersection $A\cap B$ either has an empty position or is also a type in $S$.
\end{lemma}
\begin{proof}
Let $A,B$ be two types that intersect non-trivially in every position.
It is an easy exercise to verify that $\compgr A{A\cap B}$ is a subgraph of $\compgr AB$.  Hence $\compgr A{A\cap B}$ is acyclic since $\compgr AB$ is so. By Lemma \ref{lem:refinement_iff_acyclic} this implies that $A\cap B$ is a type.
\end{proof}

\begin{lemma}\label{lem:mixsd_refinements}
Given a Minkowski cell $Q=\sum_{i=1}^k F_i$ in a \dnmixsd then the faces of $Q$ are exactly the $\mc C_R$ where $R$ is a refinement of $\mc T_Q$.
\end{lemma}
\begin{proof}
This follows directly from \cite[Proposition 6.4]{Ardila/Develin}.
\end{proof}

\begin{lemma}\label{lem:mixsd_comparable}
Let $A,B$ be $(n,d)$-types such that $\compgr AB$ is acyclic. Then $\mc C_A\cap \mc C_B=\mc C_{A\cap B}$.
\end{lemma}
\begin{proof}
It is easy to see that the intersection of the cells $\mc C_A$ and $\mc C_B$ is always the convex hull of integral points (in the standard embedding into $\R^d$) in $\dilsimp n{d-1}$.
Moreover, it is clear that $\mc C_{A\cap B}\subseteq \mc C_A\cap \mc C_B$.

Conversely, let $p$ be an integral point in $\mc C_A\cap \mc C_B$.
Denote by $p_A\subseteq A$ a possible type of $p$ (which need not be a refinement of $A$), \ie $p_A$ is an $(n,d)$-type with $p=\mc C_{p_A}$.
We will now argue that then also $p_A\subseteq B$. So suppose this is not true.
Define $p_B$ similarly to $p_A$. Then $p_B$ is a permutation of $p_A$. Hence $\compgr{p_A}{p_B}$ contains a directed cycle $C$.

But then $C$ is also contained in $\compgr AB$ (where some directed edges in $\compgr{p_A}{p_B}$ may be undirected in $\compgr AB$). But since $p_A\not\subseteq B$ there is at least one directed edge. This contradicts the hypothesis that $\compgr AB$ is acyclic.
\end{proof}

We can define the concepts of \defn[deletion!in a \mixsd]{deletion} and \defn[contraction!in a \mixsd]{contraction} for \mixsds analogous to Definition \ref{def:tom_deletion_contraction}.
The following observations are immediate:
\begin{lemma}\label{lem:mixsd_deletion_contraction}
Let $S$ be a \dnmixsd.
\begin{enumerate}
\item For any $i\in[n]$ the deletion $\deletion Si$ is a \mixsd of $\dilsimp{(n-1)}{d-1}$.
\item For any $j\in[d]$ the contraction $\contraction Sj$ is a \mixsd of $\dilsimp{n}{d-2}$.
\end{enumerate}
\end{lemma}
\begin{proof}$ $
\begin{enumerate}
\item  This follows immediately from Santos \cite[Lemma 2.1]{Santos03}.
\item The contraction $\contraction Sj$ is the subdivision of the $j$-th facet of $\dilsimp n{d-1}$ (\ie the facet opposite to the vertex $(j,\ldots,j)$) induced by $S$. Hence $\contraction Sj$ is a \mixsd.\qedhere
\end{enumerate}
\end{proof}

\smallskip
There is a standard embedding of a \dnmixsd into $\R^d$ (by mapping a tope $v$ to $(x_1,\ldots,x_d)$ where $x_i$ is the number of occurences of $i$ in $v$). We thus regard a \mixsd\ -- or any subset of its (open) cells -- as a metric  space with the Euclidean metric inherited from $\R^d$. The following is immediate:

\begin{lemma}\label{lem:deletion_embedding}
Let $S$ be a \dnmixsd, $i\in[n],j\in[d]$. Let $X$ be the subcomplex of $S$ of all cells $A$ such that $A_i=j$. Then $X$ is embedded isometrically into the deletion $\deletion Si$.
\end{lemma}

\subsection{Reconstructing Mixed Subdivisions} 
In this section we prove the following:
\begin{proposition}\label{prop:reconstruct_mixsd}
Let $S$ be a \dnmixsd. Then $S$ can be reconstructed from its topes.

More precisely, the  cells of $S$ are exactly the unions of topes all of whose total refinements are topes and which do not contain any other tope.
\end{proposition}

We call types satisfying the conditions above the \defn[nice type in a \mixsd]{nice} types of $S$. \Ie an $(n,d)$-type $A$ is \mbox{nice if}
\begin{itemize}
\item $A$ is a (componentwise) union of topes of $S$,
\item all total refinements of $A$ are topes of $S$, and
\item if $T$ is a tope of $S$ such that $T\subseteq A$ then $T$ is a refinement of $A$.
\end{itemize}  
If $A$ is a nice type we call the Minkowski cell $\mc C_A$ corresponding to $A$ a \defn[]{nice cell}.

Note that it is crucial to consider the topes of $S$ as types rather than as mere coordinates; \ie the order of the summands does matter.

Also note that the equivalent result for \toms, namely that a \tom is uniquely determined by its topes, is proven in \cite{Ardila/Develin}. Their proof, however, uses the elimination property.

\begin{figure}[b]
\centering
\begin{tikzpicture}[scale=1.3]
\coordinate (0) at (0,0);
\node[above] at (0.north) {\footnotesize $T$};
\foreach \i in {1,2,3,4,5,6}{
	\coordinate (\i) at (60*\i:1);
};

\draw[dashed] (6)--(0)--(4);
\fill[fill=white!90!black] (6)--(0)--(4)--(5)--cycle;
\draw (6)--(1)--(2)--(3)--(4)--(5)--cycle;

\end{tikzpicture}
\caption[Illustration for the proof of Proposition \ref{prop:reconstruct_mixsd}.]{A (hexagonal) Minkowski cell $A$ of type $(12,23,13)$ and a cell $B$ of type $(12,2,13)$ as in the proof of Proposition \ref{prop:reconstruct_mixsd}. Then $A\supset B$ and there is a tope $T=(1,2,3)$ of $B$ that lies in the interior \mbox{of $A$}.}
\label{fig:reconstruct_mixsd}
\end{figure}
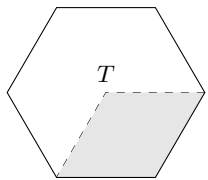

\begin{proof}
Let $S$ be a \dnmixsd.
It is clear that all cells of $S$ are nice. So it remains to prove that every nice type does indeed yield a cell of $S$.

The general strategy is the following: Assume that a cell $A$ corresponds to a nice type of $S$. We proceed via induction over $\dim A$. If $\dim A=0$ then it is clear that $A$ is a cell of $S$ (namely a vertex). Thus, we may assume that $\dim A\geq1$ and that every proper refinement of $A$ is a cell.

We will argue that $A$ intersects every cell $B$ of $S$ either not at all or in a common face of $A$ and $B$, proving that $A$ is in fact a cell in $S$.

\medskip
We may \twlog assume that $A$ contains all elements of $[d]$. Otherwise  form contractions of $S$ for each element of $[d]$ that is not contained in $A$.
Moreover, we may assume that $A$ does not contain any singleton position. Otherwise form the deletion of $S$ for every singleton position. By Lemma \ref{lem:deletion_embedding} $A$ embeds isometrically into this deletion.

Now let $B$ be a cell in $S$. By Lemma \ref{lem:mixsd_comparable} it suffices to prove that $A$ and $B$ are comparable. So suppose on the contrary that  $\compgr BA$ has a directed cycle.

Assume \twlog that this cycle is $C=(1,2,\ldots,k,1)$, directed in this order.
Let $P=([k],k+1,\ldots,d)$ be an ordered partition of $[d]$. Define $A'\coloneq\restr AP$. Since $A$ does not have any singleton positions, $\dim A'<\dim A$ if $k<d$ and hence $A'$ is a proper refinement of $A$. Moreover, $\compgr B{A'}$ also contains the cycle $C$. This is a contradiction. 

Thus, $k=d$. Assume \twlog that $B_i\ni i$ and $A_i\ni (i+1)\bmod d$ for each $i$. Since $A$ does not have any singleton positions this implies that $A_i=\{i,i+1\bmod d\}$ for each $i$. Moreover, $B_i=\{i\}$ if there is a directed edge $i\to (i+1\bmod d)$ and $B_i=\{i,i+1\bmod d\}$ if the edge is undirected. Thus, we have completely determined $A$ and $B$.

Since the cycle is directed, there is a singleton in $B$. Assume \twlog that $B_d=\{d\}$. Let $P=(1,2,\ldots,d)$ be an ordered partition of $[d]$ into singletons. Then $T\coloneq\restr BP=(1,2,\ldots,d)$. Hence $T$ is a tope in $S$. But $T$ is contained in $A$ and not a refinement of $A$. This contradicts the choice of $A$. See Figure \ref{fig:reconstruct_mixsd} for an illustration.
\end{proof}

Since in a \emph{fine} \mixsd the type graph of every type is acyclic, we get the following:
\begin{corollary}
Let $S$ be a fine \dnmixsd. Then the (type graphs of the) cells of $S$ are exactly the acyclic unions of (the type graphs of) topes all of whose total refinements are again topes.
\end{corollary}

\bigskip
For $i\in[n]$ consider the \emph{deletion map} \[\deletion {\cdot\,}i:S\to \deletion Si: C\mapsto \deletion Ci=(C_1,\ldots,\omitted{C_i},\ldots, C_n)\] mapping each cell $C$ of $S$ to the cell obtained by omitting the $i$-th entry of $C$.

\begin{lemma}\label{lem:deletion_preimage_connected}
Let $S$ be a \dnmixsd, $i\in[n]$ and $A\neq B$ the types of cells $\mc C_A,\mc C_B\in S$ cells such that $\deletion Ai=\deletion Bi$.  Then $A\cup B$ is the type of a cell in $S$.
\end{lemma}

\begin{proof}
Let $C\coloneq A\cup B$, \ie $C_i\coloneq A_i\cup B_i$ and $C_j=A_j(=B_j)$ for $j\neq i$. The situation is sketched in Figure \ref{fig:deletion}. The intuition is that  $C$ (unless it  already equals $A$ or $B$) is a prism over $A$ (or $B$) with $A$ and $B$ the top, respectively bottom face of $C$.

We need to show that $C$ is indeed a cell in $S$. To this end, we verify that $C$ satisfies the conditions from Proposition \ref{prop:reconstruct_mixsd}. This means we have to show that the total refinements of $C$ are exactly the total refinements of $A$ and $B$.

Indeed let $v=\restr CP$ be a total refinement of $C$ and assume \twlog that $v_i\in A_i$. Then $v=\restr AP$ is also a total refinement of $A$. 
Conversely, let $v=\restr AP$ be a total refinement of $A$. We may assume that in $P$ the element $v_i$ comes before all elements of $B_i\setminus A_i$. Otherwise we may change this order since $\compgr AB$ is acyclic. 
But then $\restr CP= v.$
Thus, every total refinement of $A$ or $B$ is also one of $C$. 
Hence $C$ is a type in $S$.
\end{proof}

\begin{figure}[h]
\centering
\includegraphics[width=3cm]{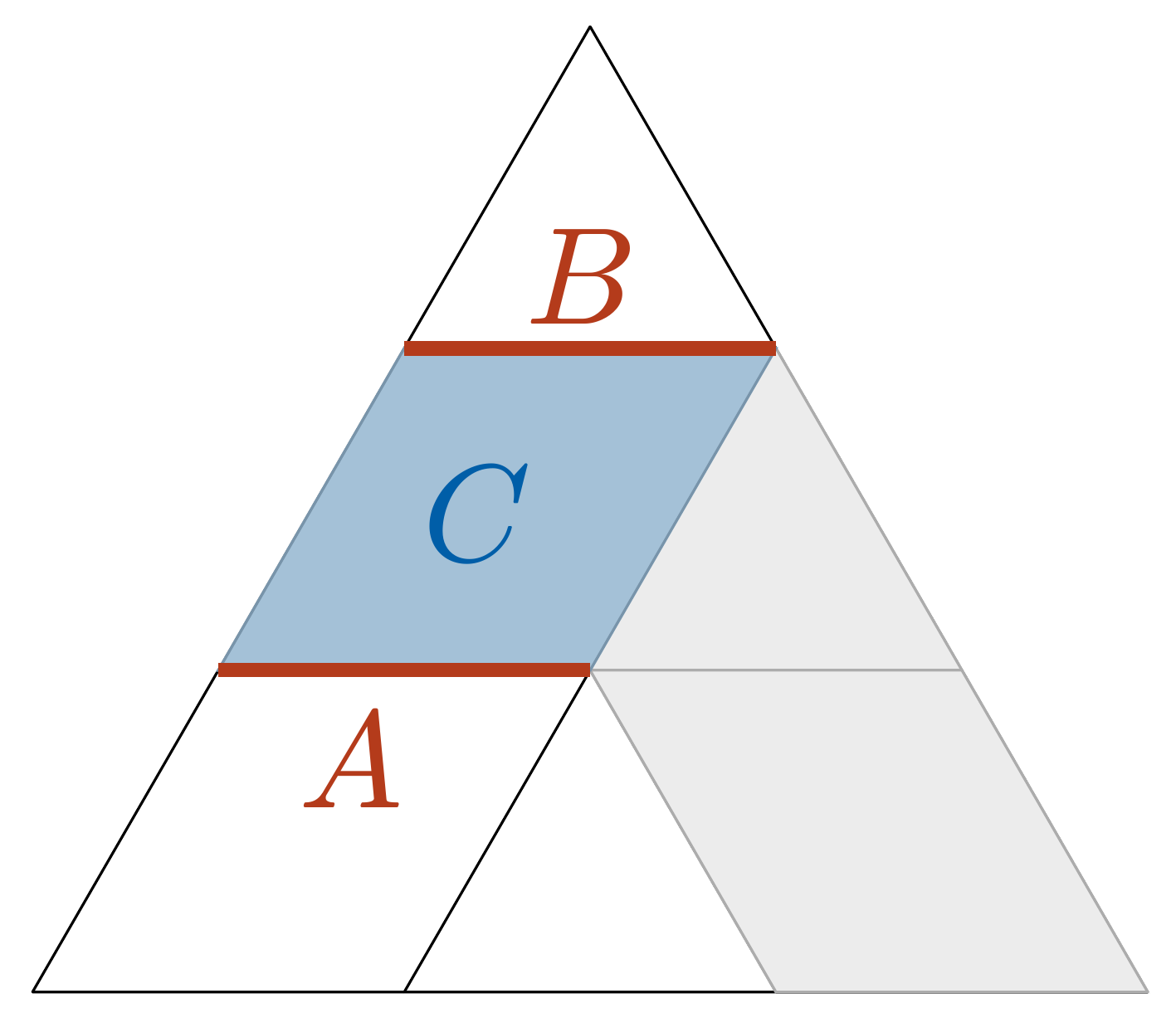}
\caption{The two edges $A$ and $B$ are mapped to the same cell under the deletion map that deletes the shaded cells.}
\label{fig:deletion}
\end{figure}

\section{The ``Tropological Representation Theorem''}\label{sec:toprep}
In this section we  formally introduce \tphps and prove a first version of the Topological Representation Theorem. 

\begin{definition}\label{def:tphp}
A \defn{tropical pseudohyperplane} is the image of a \thp under a PL-homeo\-mor\-phism of $\TP^{d-1}$ that fixes the boundary.
\end{definition}

The following theorem is a crucial ingredient to the proof of the Topological Representation Theorem.
In an arrangement of \thps, the $i$-th \thp consists exactly of those points $A$ with $\card A_i\geq 2$. We show that the analogue holds for the Poincaré dual of a \dnmixsd. We denote the dual cell of a cell $C\in S$ by $\dual C$.
See again Figure \ref{fig:mixsd2thp_b} for an example.

\begin{theorem}\label{thm:dual_phps}
Let $S$ be a \dnmixsd and $i\in [n]$. Then $\{\dual C\mid C\in S,\card{C_i}\geq2\}$ is a \tphp.
\end{theorem}

\begin{proof}
We  prove the claim by induction over $n$. For $n=1$ this is true since then $S=\nsimplex{d-1}$ is the trivial subdivision, whose dual is the cell complex of one $(d-2)$-dimensional \thp in $\T^{d-1}$.

Now assume $n\geq 2$. Choose $i\neq j\in[n]$ and consider the deletion $\deletion Sj$. By Lemma \ref{lem:mixsd_deletion_contraction} this is a \mixsd of $\dilsimp{(n-1)}{d-1}$ and by induction the image of $H_i$ in $\deletion Sj$ is a \tphp $h$.

But $H_i$ is the preimage of $h$ under the deletion map.  
By Lemma \ref{lem:deletion_preimage_connected} this preimage is PL-homeomorphic to $h$ and hence a \tphp.
\end{proof}

\subsection{Arrangements of \tphps I}
In this section we suggest one definition for \tphpas. Note that another (equivalent) definition is given in \cite{toprep2}.
\begin{definition}\label{def:tphpa1}
An \defn[arrangement!of tropical!pseudohyperplanes]{arrangement of tropical pseudohyperplanes} is a finite family of \tphps such that
\begin{itemize}
\item the \tphps induce a regular subdivision of $\T^{d-1}$,
\item in the cell decomposition the points of equal type form 
a PL-ball\index{PL!ball} 
(in particular, there are no two cells with the same type),
\item the types satisfy the surrounding and comparability property\index{surrounding axiom}\index{comparability!axiom} and
\item the bounded cells are exactly those which correspond to bounded types. \index{bounded!type}\index{bounded!cell}
\end{itemize}
\end{definition}

The following theorem is the main result of this paper and can be seen as a first version of the Topological Representation Theorem for \toms.
\begin{theorem}[{\defn[Topological Representation Theorem!for tropical oriented matroids]{Tropological Representation Theorem}, Version I}]\label{thm:toprep1}
Let $n,d\geq 1$. The Poincaré dual of a \mixsd of $\dilsimp n{d-1}$ is a \tphpa as defined in Definition \ref{def:tphpa1}. Conversely, the dual of the cell decomposition of an arrangement of $n$ \tphps in $\TP^{d-1}$ is a \dnmixsd.
\end{theorem}

\begin{proof}
Let $S$ be a \dnmixsd. By Theorem \ref{thm:dual_phps} and \cite[Proposition 6.4]{Ardila/Develin}, it is clear that $S$ satisfies the axioms in Definition \ref{def:tphpa1} above. 

Conversely, let $\mc A$ be an arrangement of \tphps in $\T^{d-1}$ as in Definition \ref{def:tphpa1}. We have to show that the types of the cells in the induced cell decomposition yield a \dnmixsd.
So let $S\coloneq \{\mc C_A\mid A \text{ type in the cell complex of }\mc A\}$. Then $S$ is a set of Minkowski cells in $\dilsimp n{d-1}$.

By Lemmas \ref{lem:mixsd_refinements} and \ref{lem:mixsd_comparable}, $S$
is a polytopal complex whose realisation is contained in $\dilsimp n{d-1}$. It remains to show that $S$ covers $\dilsimp n{d-1}$. We will use the fact that the $1$-skeleton of $\mc A$ is path-connected.

To this end, let $\mc C_A$ be a maximal cell in $S$ and let $\mc C_B$ be a facet of $\mc C_A$. Then $A$ corresponds to a vertex in $\mc A$ and $B$ corresponds to an edge containing $A$.  The cell $\mc C_B$ is contained in the boundary of $\dilsimp n{d-1}$ \tiff $B$ is unbounded. In this case $B$ is an unbounded edge in $\mc A$. If $\mc C_B$ is not on the boundary then there is a unique other maximal cell $\mc C_{A'}$ ``on the other side'' of $\mc C_B$, the other endpoint of $B$. Thus, $S$ covers the whole of $\dilsimp n{d-1}$.
\end{proof}

\nocite{mydiss}
\printbibliography
\end{document}